\documentclass[12pt]{article}
 \usepackage[margin=1in]{geometry} 
\usepackage{amsmath,amsthm,amssymb,amsfonts}
 \usepackage{amssymb}

\usepackage[pagewise]{lineno}
\usepackage[utf8]{inputenc}

\usepackage{hyperref}
 \usepackage{csquotes}
\usepackage[utf8]{inputenc}
\usepackage[english]{babel}
\usepackage{xcolor}

\usepackage{enumitem}

\numberwithin{equation}{section}
\newtheorem{theorem}{Theorem}[section]
\newtheorem{lemma}[theorem]{Lemma}
\newtheorem{remark}{Remark}[section]

\newtheorem{corollary}[theorem]{Corollary}
\providecommand{\keywords}[1]
{
  \small	
  \textbf{\textit{Keywords:}} #1
}

\newcommand{\MSC}[1]{%
  \small
  \textbf{\textit{Mathematics Subject Classification:}} #1
}
\title{ An improvement toward global boundedness in a fully parabolic chemotaxis with singular sensitivity in any dimension    }
\author{
    Minh Le\thanks{ Institute For Theoretical Sciences, Westlake University, China \texttt{(leminh@westlake.edu.cn)}} }
\date{}

\begin{document}
\maketitle

\begin{abstract}
This paper deals with the problem of global solvability and boundedness of classical solutions to a fully parabolic chemotaxis system with singular sensitivity in any dimensional setting. In particular, We show that the system
\begin{equation*} 
\begin{cases}
u_t = \Delta u - \chi \nabla \cdot \left( \dfrac{u}{v} \nabla v \right), \\
v_t = \Delta v - v + u,
\end{cases}
\end{equation*}
posed in a bounded domain $\Omega \subset \mathbb{R}^n$ with $n \geq 3$, admits a global bounded classical solution provided that $\chi \in (0,\chi_0)$ with $\chi_0 > \sqrt{\frac{2}{n}}$ can be determined explicitly. This result extends several existing works, which established global boundedness under the more restrictive condition $\chi < \sqrt{\frac{2}{n}}$, and shows that this threshold is not an optimal upper bound for preventing blow-up.

\end{abstract}
\\
\MSC{35B35, 35K45, 35K55, 92C15, 92C17}

\numberwithin{equation}{section}

\newtheorem{Corollary}{Corollary}[theorem]
\allowdisplaybreaks

\section{Introduction}
We consider the following system arising from a chemotaxis model in a bounded domain $\Omega \subset \mathbb{R}^n$, where $n \geq 3$:
\begin{equation}  \label{1}
    \begin{cases}
        u_t = \Delta u -\chi \nabla \cdot \left (  \frac{u}{v} \nabla v \right ) , \qquad &\text{in } \Omega \times (0,T_{\rm max}),\\
         v_t = \Delta v - v+u, \qquad &\text{in } \Omega \times (0,T_{\rm max}),
    \end{cases}
\end{equation}
where $\chi>0$ and $T_{\rm max} \in (0,\infty]$ is the maximal existence time. The system is endowed with the homogeneous Neumann boundary condition
\begin{equation} \label{bdry}
    \frac{\partial u}{\partial \nu} =\frac{\partial v}{\partial \nu} =0 ,\qquad (x,t) \in \partial \Omega \times (0, T_{\rm max}),
\end{equation}
and initial conditions
\begin{equation} \label{initial}
    \begin{cases}
        u_0 \in C^{0} (\bar{\Omega}) \text{ is nonnegative with } \int_\Omega u_0 >0,  \\
      v_0 \in W^{1, \infty}(\Omega), \quad v_0> 0 \quad \text{in } \bar{\Omega}.
    \end{cases}
\end{equation}
 The system \eqref{1} was originally introduced in the seminal work \cite{Keller}, where it was proposed as a mathematical model to describe the phenomenon of chemotaxis—namely, the directed movement of cells or microorganisms toward or away from chemical signals in their environment. This mechanism plays a fundamental role in a variety of biological processes, such as the aggregation of cellular slime molds, immune response behavior, and the formation of patterns in bacterial populations.
In the context of system \eqref{1}, the function \( u(x,t) \) typically denotes the density of a cell population (such as bacteria or amoebae), while \( v(x,t) \) represents the concentration of a chemical substance that influences the movement of these cells at a given spatial location \( x \) and time \( t \).

\subsection{Literature review}

In this subsection, let us briefly recall some mathematical results associated with the system \eqref{1}. 

A well-studied variant of system \eqref{1} arises when the second equation is replaced by the elliptic equation
\[
0 = \Delta v - v + u,
\]
leading to the so-called parabolic--elliptic chemotaxis system. Under the assumption of radially symmetric initial data, it was shown in \cite{NaSe3} that solutions exist globally in time provided that the chemotactic sensitivity satisfies
\(
\chi < \frac{2}{n-2},
\)
whereas finite-time blow-up may occur if $\chi > \frac{2n}{n-2}$. This result was subsequently strengthened in \cite{Biler-1999}, where it was shown that the condition $\chi < \frac{2}{n}$ guarantees the global existence of solutions for arbitrary initial data. Moreover, global boundedness of solutions under this assumption was established in \cite{KMT}.

In two spatial dimensions, a substantially stronger result was obtained in \cite{Fujie+Senba}, where it was shown that blow-up is excluded for all $\chi>0$, thereby revealing a fundamental qualitative distinction between two-dimensional and higher-dimensional dynamics. Very recently, the author in \cite{Kurt+2025} further improved these results in higher dimensional domains by showing that the parabolic-elliptic system \eqref{1} possesses a global bounded classical solution provided that
\[
\chi < \frac{2}{n} + \frac{2n-1}{2n^3} \sqrt{\frac{n}{2n+2}},
\qquad n \ge 3.
\]

Despite these substantial advances, the precise critical value of $\chi$ separating global existence from finite-time blow-up in the parabolic--elliptic setting remains an open problem. For further developments concerning parabolic--elliptic chemotaxis systems with singular sensitivity, we refer the reader to \cite{zh, Zhao1, Minh6, Kurt, Kurt+Shen, minh-kurt}.

We now turn our attention to the fully parabolic case. 
In two spatial dimensions, it was shown that the condition $\chi \leq 1$ guarantees the global solvability of solutions to system~\eqref{1} \cite{Biler-1999}. 
This threshold was later slightly improved, and blow-up was excluded whenever $\chi < 1.015$ \cite{Lankeit_2016}. 
Moreover, under the additional assumption that the initial data are radially symmetric, global boundedness was obtained for a modified system in which the left-hand side of the second equation is replaced by $\tau v_t$, provided that $\tau>0$ is sufficiently small; in this setting, solutions remain global and bounded for arbitrary $\chi>0$ \cite{Fujie+Senba-1}.

In higher-dimensional domains, global existence of solutions was established, and global boundedness was further derived under the condition
\[
\chi < \sqrt{\frac{2}{n}}
\]
\cite{Winkler+2011, Fujie}. 
If the left-hand side of the first equation is replaced by $\tau u_t$ with sufficiently small $\tau>0$, then global existence and boundedness hold for all
\[
\chi < \frac{n}{n-2}
\]
\cite{Fujie+Senba}. In contrast, when $\tau$ is sufficiently small and $\chi > \frac{n}{n-2}$, solutions may become arbitrarily large \cite{Winkler+2022_JDE}. More recently, without imposing any smallness assumption on $\tau$, it was proved in \cite{Ahn+Kang+Lee} that the condition
\[
\chi < \frac{4}{n}
\]
ensures global existence and boundedness of solutions, thereby improving earlier results \cite{Winkler+2011, Lankeit_2016}. Beyond classical solutions, the existence of global weak solutions was obtained under the restriction
\[
\chi < \sqrt{\frac{n+2}{3n-4}}
\]
\cite{Winkler+2011}. By introducing a slightly more general notion of weak solutions, global existence of radially symmetric solutions was further established provided that
\[
\chi < \sqrt{\frac{n}{n-2}}
\]
\cite{Winkler+Stinner_2011}. Finally, the existence of global generalized solutions was shown under the assumption
\[
\chi <
\begin{cases}
\infty, & \text{if } n = 2,\\[0.3em]
\sqrt{8}, & \text{if } n = 3,\\[0.3em]
\dfrac{n}{n-2}, & \text{if } n \ge 4,
\end{cases}
\]
see \cite{Winkler+Lankeit_2017}. For further developments concerning solvability and long-time behavior of variants of system~\eqref{1}, we refer the reader to 
\cite{Zhao2, Winkler_Yokota_2018, Le_2026, Le_2025_1, Le_2025_2, Li_Xie_2024, Li_Xie_2026}.

\subsection{Motivation and Contribution}
In two spatial dimensions, the author in \cite{Biler-1999} investigated the following energy functional to obtain global existence of solutions to \eqref{1} 
\[
A(t)
= \int_\Omega u \ln u - a \int_\Omega u \ln v
\qquad \text{for all } t>0,
\]
where $a>0$. This result was later refined in \cite{Lankeit_2016}, where the author demonstrated that global existence and boundedness by employing an additional term $\int_\Omega \frac{|\nabla v|^2}{v}$ to the functional $A(t)$. More precisely, the functional
\[
B(t)
= \int_\Omega u \ln u
- a \int_\Omega u \ln v
+ b \int_\Omega \frac{|\nabla v|^2}{v}
\qquad \text{for all } t>0,
\]
with $b>0$, was investigated to obtain a stronger threshold for $\chi$.

In higher dimensions ($n \ge 3$), as mentioned before, \cite{Winkler+2011} established the existence of a global classical solution under the assumption $\chi < \sqrt{\frac{2}{n}}$. The proof relied on the functional
\[
C(t)
= \int_\Omega u^p v^{-r}
\qquad \text{for all } t>0,
\]
where $p>1$ and $r>1$. This naturally raises a question analogous to that posed in \cite{Lankeit_2016}: namely, whether the approach developed in \cite{Winkler+2011} already yields a sharp criterion for global boundedness, or whether it can be refined in a manner similar to the two-dimensional case.

Motivated by the aforementioned results, we propose a new approach showing that the value $\sqrt{\frac{2}{n}}$ is not the sharp upper bound for blow-up prevention in arbitrary spatial dimensions. To this end, we introduce the energy functional
\[
F_\lambda(u,v)
= \int_\Omega u^p v^{-r}
+ \lambda \int_\Omega \frac{|\nabla v|^{2p}}{v^{p+r}}
+ \int_\Omega v^{p-r}
\qquad \text{for all } t>0,
\]
where $p>1$, $r = \frac{p-1}{2}$, and $\lambda>0$.

The inclusion of the additional term
\[
\int_\Omega \frac{|\nabla v|^{2p}}{v^{p+r}}
\]
provides extra flexibility in the energy estimates and allows us to extend the admissible upper bound on $\chi$ beyond that obtained in \cite{Winkler+2011}. The construction of this term is inspired by \cite{Winkler+2022}, where an energy functional containing
\[
\int_\Omega \frac{|\nabla v|^q}{v^{q-1}}
\]
was introduced to address chemotaxis-consumption systems with weakly singular sensitivity.

In addition, the presence of the term $\int_\Omega v^{p-r}$ in the functional $F_\lambda(u,v)$ enables us to derive a direct bound for $\int_\Omega u^p v^{-r}$, thereby avoiding the iterative arguments used in \cite{Fujie,MT_2017} and eliminating the need for any a priori estimate on $\int_\Omega v^{p-r}$.

\subsection{Main result}
Our main result on the global existence and boundedness of solutions to system \eqref{1} can now be stated as follows.

\begin{theorem} \label{thm}
    Let $\Omega \subset \mathbb{R}^n$, with $n \geq 3$, be a bounded domain with smooth boundary. Then there exists $\chi_0 > \sqrt{\frac{2}{n}}$ such that, for any $\chi \in (0, \chi_0)$, the system \eqref{1}, with the boundary condition \eqref{bdry} and initial condition \eqref{initial}, admits a unique global positive classical solution $(u,v) \in \left[ C^{2,1}\left( \bar{\Omega} \times (0,\infty) \right) \cap C^0 \left( \bar{\Omega} \times [0,\infty) \right) \right]^2$ such that 
    \begin{align*}
       \sup_{t>0} \left\| u(\cdot,t) \right\|_{L^\infty(\Omega)} < \infty.
    \end{align*}
\end{theorem}

\begin{remark} \label{Rm}
An explicit lower bound for $\chi_0$ is established in Lemma~\ref{L4}, namely
\[
\chi_0 \ge \sqrt{\frac{2(1+\delta)}{n}},
\]
where
\[
\delta = \frac{L_2 L_4}{2 L_3 L_5},
\]
with 
\begin{align*}
    L_2&=  \frac{n-2}{16}, \qquad L_3= 16^n \left ( \frac{(2n-1)(3n-1)}{2} \right )^{n+1} + (304 n(2n-2+\sqrt{n}) )^{n+1},\\
    L_4 &= \frac{n-2}{8}, \qquad L_5 = \frac{n-1}{4}.
\end{align*}
Moreover, one can verify that
\[
\delta > \frac{(n-2)^2}{256 (n-1)\left( 16^n (3n^2)^{n+1} + (912 n^2)^{n+1} \right)}.
\]

\end{remark}
\begin{remark}
The works \cite{Lankeit_2016, Fujie_2018, Ahn+Kang+Lee} require the convexity of the domain $\Omega$, whereas Theorem~\ref{thm} imposes no such assumption. 
This is due to Lemma~\ref{Lbe}, which enables control of the nonlinear boundary integrals through the diffusion term without relying on the convexity of $\Omega$.

\end{remark}
From a biological perspective, the result of Theorem~\ref{thm} indicates that when chemotactic sensitivity is weak, the combined effects of diffusion and signal degradation are sufficient to regulate cell movement, thereby preventing excessive aggregation in any spatial dimension. This mechanism stands in sharp contrast to classical chemotaxis systems, where diffusion alone cannot balance chemotactic accumulation, often leading to pronounced clustering of the population.

\textbf{Future work. }An important open question is the optimal choice of $(r, \lambda)$ that maximizes the value of $\chi_0$, ensuring that for any $\chi < \chi_0$ the solutions exist globally and remain uniformly bounded in time. Another interesting direction for future research is to explore applications of the boundedness of 
\[
\int_\Omega \frac{u^p}{v^r}+ \int_\Omega \frac{|\nabla v|^{2p}}{v^{p+r}}
\] 
in preventing blow-up for chemotaxis systems with (weak) singular sensitivity and logistic growth, or for chemotaxis-fluid models with singular sensitivity. In such models, as can be seen, a uniform lower bound for $v$ is generally no longer guaranteed due to the presence of logistic sources or fluid interactions. We leave these questions as topics for future investigations.

The structure of the paper is as follows: In Section \ref{S2}, we recall the local well-posedness of and an extensibility property for solutions to \eqref{1} and establish several auxiliary inequalities that will be used in the subsequent analysis. Section \ref{S3} provides several identities and estimates for solutions. In the final section, we prove the time-independent  $L^p$ boundedness of solutions and then complete the proof of the main result.

\section{Preliminaries} \label{S2}

Let us begin this section by recalling the local existence and an extensibility criterion of solutions to the system \eqref{1} as established in \cite{Lankeit_2016}[Theorem 2.3].
\begin{lemma} \label{local}
    Let $n \geq 1$, $\Omega \subset \mathbb{R}^n$ a bounded, smooth domain and $q>n$. Then there exist $T_{\rm max} \in (0,\infty]$ and a unique pair of nonnegative functions $(u,v)$ such that
    \begin{align*}
        u &\in C^0 \left(\bar{\Omega} \times [0,T_{\rm max})\right) \cap C^{2,1} \left(\bar{\Omega} \times (0,T_{\rm max})\right) \quad \text{and }\notag \\
        v &\in C^0 \left(\bar{\Omega} \times [0,T_{\rm max})\right) \cap C^{2,1} \left(\bar{\Omega} \times (0,T_{\rm max})\right)\cap L^\infty_{loc} \left ([0,T_{\rm max}); W^{1,q}(\Omega) \right ),
    \end{align*}
    solving the system \eqref{1} with initial conditions \eqref{initial} and boundary conditions \eqref{bdry}. Moreover, $u>0$ and $v>0$ in $\bar{\Omega} \times (0, T_{\rm max})$, and 
    \begin{align} \label{ext}
        \text{if } T_{\rm max}< \infty, \text{ then } \limsup_{t \to T_{\rm max}} \left \{ \left \| u(\cdot,t) \right \|_{L^\infty(\Omega)}+\left \| v(\cdot,t) \right \|_{W^{1,q}(\Omega)} \right \} = \infty.
    \end{align}
\end{lemma}
From now on, we denote $(u,v)$ as a classical solution to \eqref{1} in $(0,T_{\rm max})$. The following lemma gives a uniform in time $L^1$ bound for $v$.
\begin{lemma} \label{L1}
    There exists positive constants $m$ and $C$ such that 
    \begin{align*}
         \int_\Omega u(\cdot,t)=m \qquad \text{for all }t\in (0,T_{\rm max}),
    \end{align*}
    and
    \begin{align*}
     \int_\Omega v(\cdot,t) \leq C\qquad \text{for all }t\in (0,T_{\rm max}).
    \end{align*}
\end{lemma}
\begin{proof}
    By integrating the first equation of \eqref{1} over $\Omega$, we obtain that 
    \begin{align*}
        \frac{d}{dt}\int_\Omega u(\cdot,t) =0 \qquad \text{for all }t\in (0,T_{\rm max}),
    \end{align*}
    which entails that 
    \begin{align*}
        \int_\Omega u(\cdot,t)=m:= \int_\Omega u_0 \qquad \text{for all }t\in (0,T_{\rm max}).
    \end{align*}
    Integrating the second equation of \eqref{1} over $\Omega$ yields
    \begin{align*}
        \frac{d}{dt}\int_\Omega v + \int_\Omega v = m  \qquad \text{for all }t\in (0,T_{\rm max}).
    \end{align*}
    Applying Gronwall's inequality to this, it follows that 
    \begin{align*}
        \sup_{t\in (0,T_{\rm max})} \int_\Omega v(\cdot,t) \leq \max \left \{ \int_\Omega v_0, m \right \},
    \end{align*}
    which completes the proof.
\end{proof}

The next lemma gives us a time independent lower bound for $v$, which was established in \cite{Fujie_2018}[Lemma 3.1].
\begin{lemma} \label{low}
    There exists a positive constant $\eta$ depending only on $\inf_{\bar{\Omega}} v_0, \int_\Omega u_0$, and $\Omega$ such that 
    \begin{align*}
        \inf_{x\in \Omega} v(x,t) \geq \eta \qquad \text{for all }t\in (0,T_{\rm max}). 
    \end{align*}
\end{lemma}

The following lemma establishes regularity properties of \( v \) based on the boundedness of \( u \); for a detailed proof, we refer the reader to \cite{Winkler+2011}[Lemma 2.4].
\begin{lemma} \label{v}
    Let $T \leq  T_{\rm max}$ and $T < \infty$, and $1 \leq \theta, \mu \leq \infty$.
\begin{enumerate}[label=(\roman*)]
\item If $\frac{n}{2}\left( \frac{1}{\theta}- \frac{1}{\mu} \right )<1$, then there exists $C>0$ independent of $T$ such that  
\begin{align*}
    \left \| v(\cdot,t) \right \|_{L^\mu(\Omega) } \leq C\left ( 1+ \sup_{s\in (0,T)} \left \|u(\cdot,s) \right \|_{L^\theta(\Omega)} \right )\quad \text{for all }t\in (0,T).
\end{align*}
\item If $\frac{1}{2}+\frac{n}{2}\left( \frac{1}{\theta}- \frac{1}{\mu} \right )<1$, then there exists $C>0$ independent of $T$ such that  
\begin{align*}
    \left \|\nabla  v(\cdot,t) \right \|_{L^\mu(\Omega) } \leq C\left ( 1+ \sup_{s\in (0,T)} \left \|u(\cdot,s) \right \|_{L^\theta(\Omega)} \right )\quad \text{for all }t\in (0,T).
\end{align*}
\end{enumerate}
\end{lemma}

Now we will derive several useful inequalities similar to \cite{Winkler2012}[Lemma 3.3] and \cite{Winkler+2022}[Lemma 3.4], which will be later applied in Lemma \ref{L2'}. 

     \begin{lemma} \label{Lw2}
     Let $l>0$ and $q > l+1$. For any $\phi \in C^2(\bar{ \Omega  })$ such that $\phi>0$ in $\bar{\Omega}$ and $\frac{\partial \phi}{\partial \nu}=0$ on $\partial \Omega$, the following inequalities hold
        \begin{align} \label{Lm.1}
             \int_\Omega \phi^{-l-2}|\nabla \phi|^{q+2} \leq \left ( \frac{q+\sqrt{n}}{l+1} \right )^2 \int_\Omega \phi^{-l}|\nabla \phi|^{q-2} |D^2 \phi|^2,
        \end{align}
        and
    \begin{align} \label{Lm.2}
        \int_\Omega \phi^{-l-2}|\nabla \phi|^{q+2} \leq \left ( 4\left ( \frac{l+1+\sqrt{n}}{q-l+1} \right )^2+2 \right )\int_\Omega \phi^{-l+2}|\nabla \phi|^{q-2}|D^2 \ln \phi|^2,
    \end{align}
    as well as
    \begin{align} \label{Lm.3}
        \int_\Omega \phi^{-l}|\nabla \phi|^{q-2} |D^2 \phi|^2 \leq \left ( \frac{q-l+1}{q-l-1} \cdot \left ( \frac{l+1+\sqrt{n}}{q-l+1} \right )^2+1 \right )\int_\Omega \phi^{-l+2}|\nabla \phi|^{q-2}|D^2 \ln \phi|^2.
    \end{align}
    \end{lemma}
\begin{proof}
    Setting 
    \begin{align*}
        A&=  \int_\Omega \phi^{-l-2} |\nabla\phi|^{q+2}, \\
        B&=\int_\Omega \phi^{-l}|\nabla \phi|^{q-2} |D^2 \phi|^2,\\
        C&=  \int_\Omega \phi^{-l+2}|\nabla \phi|^{q-2}|D^2 \ln \phi|^2, \\
        D&=\int_\Omega  \phi^{-l-1}|\nabla \phi|^{q-2} \nabla \phi \cdot D^2 \phi \cdot \nabla \phi. 
    \end{align*}
    and employing integration by parts yields 
    \begin{align*}
       A &=- \int_\Omega \nabla \cdot \left ( \phi^{-l-2} |\nabla \phi|^q \nabla \phi \right ) \phi \notag \\
        &= (l+2)  \int_\Omega \phi^{-l-2} |\nabla\phi|^{q+2} - \int_\Omega \phi^{-l-1} \nabla |\nabla \phi|^q \cdot \nabla \phi - \int_\Omega \phi^{-l-1}|\nabla \phi|^q \Delta \phi. 
    \end{align*}
    By using the identity $\nabla |\nabla \phi|^2 = 2 D^2\phi \cdot \nabla \phi$ and the inequality $|\Delta \phi| \leq \sqrt{n} |D^2 \phi|$, we obtain that 
    \begin{align*}
        (l+1)A &= q \int_\Omega \phi^{-l-1}|\nabla \phi|^{q-2} \nabla \phi \cdot D^2 \phi \cdot \nabla \phi + \int_\Omega \phi^{-l-1}|\nabla \phi|^q \Delta \phi \notag \\
        &\leq (q+ \sqrt{ n})\int_\Omega \phi^{-l-1}|\nabla \phi|^q |D^2 \phi|  \notag \\
        &\leq (q+ \sqrt{ n}) \sqrt{AB},
    \end{align*}
    which further implies \eqref{Lm.1}. We use integration by parts and Hölder's inequality to derive \eqref{Lm.2} 
    \begin{align} \label{M.1}
        A&= \int_\Omega |\nabla \ln \phi|^{l+1} |\nabla \phi|^{q-l-1}\nabla \ln \phi \cdot \nabla \phi  \notag \\
        &=- \int_\Omega \nabla |\nabla \ln \phi|^{l+1} \cdot \nabla \ln \phi |\nabla \phi|^{q-l-1}\phi- \int_\Omega \phi |\nabla \ln \phi|^{l+1}|\nabla \phi|^{q-l-1}\Delta \ln \phi  \notag \\
        &\quad-\int_\Omega  \phi |\nabla \ln \phi|^{l+1}\nabla \ln \phi \cdot \nabla |\nabla \phi|^{q-l-1} \notag \\
        &= -(l+1)\int_\Omega  \phi^{-l}|\nabla \phi|^{q-2} \nabla \phi \cdot D^2 \ln \phi \cdot \nabla \phi - \int_\Omega \phi^{-l}|\nabla \phi|^{q} \Delta \ln \phi \notag \\
        &\quad-(q-l-1) \int_\Omega  \phi^{-l-1}|\nabla \phi|^{q-2} \nabla \phi \cdot D^2 \phi \cdot \nabla \phi \notag \\
        &\leq (l+1+\sqrt{n}) \int_\Omega \phi^{-l}|\nabla \phi|^q |D^2 \ln \phi| -(q-l-1)D \notag \\
        &\leq (l+1+\sqrt{n}) \sqrt{AC} -(q-l-1)D. 
    \end{align}
    Employing the identity 
    \begin{align*}
        |D^2 \phi|^2 = \phi^2 |D^2 \ln \phi|^2 + \frac{2}{\phi} \nabla \phi \cdot \left ( D^2 \phi \cdot \nabla \phi \right ) - \frac{|\nabla \phi|^4}{\phi^2},
    \end{align*}
    we obtain 
    \begin{align} \label{M.2}
        B=C+2D-A.
    \end{align}
    Combining \eqref{M.1} and \eqref{M.2} with straightforward simplifications, we deduce that
    \begin{align} \label{M.3}
        A+\frac{q-l-1}{q-l+1}B \leq \frac{2(l+1+\sqrt{n})}{q-l+1} \sqrt{AC} + \frac{q-l-1}{q-l+1}C.
    \end{align}
    Dropping the second term of the {left} hand side and using Young's inequality, we derive 
    \begin{align*}
        A \leq \frac{A}{2} +2 \left ( \frac{l+1+\sqrt{n}}{q-l+1} \right )^2 C + C.
    \end{align*}
   This leads to 
    \begin{align}
        A \leq \left ( 4\left ( \frac{l+1+\sqrt{n}}{q-l+1} \right )^2+2 \right )C,
    \end{align}
    which proves \eqref{Lm.2}. To prove \eqref{Lm.3}, we apply Young's inequality to \eqref{M.3} to obtain that
    \begin{align*}
        A+\frac{q-l-1}{q-l+1}B \leq A + \left ( \frac{l+1+\sqrt{n}}{q-l+1} \right )^2C + \frac{q-l-1}{q-l+1}C.
    \end{align*}
    This further leads to 
    \begin{align*}
        B \leq \left ( \frac{q-l+1}{q-l-1} \cdot \left ( \frac{l+1+\sqrt{n}}{q-l+1} \right )^2+1 \right )C,
    \end{align*}
    which proves \eqref{Lm.3}. The proof is now complete.
    
\end{proof} 
The following lemma provides a bound for the boundary integral that arises in the analysis of the quantity \( \frac{d}{dt} \int_\Omega \frac{|\nabla v|^{2p}}{v^{p+r}} \), as encountered in Lemma~\ref{L2}. This allows us to work on a general domain $\Omega$ without the convexity assumption.

\begin{lemma} \label{Lbe}
    Let $p\geq 1$, $0<r<p$, and $\eta>0$. Then there exists $C=C(\eta,p,r,n,\Omega)>0$ such that the following holds
    \begin{align} \label{Lbe-1}
        \int_{\partial \Omega} \phi^{-p-r}|\nabla \phi|^{2p-2} \cdot \frac{\partial |\nabla \phi|^2}{\partial \nu }\leq \eta \int_\Omega \phi^{-p-r+2}|\nabla \phi|^{2p-2} |D^2 \ln\phi|^2+C \int_\Omega \phi^{p-r},
    \end{align}
    for any $\phi \in C^2(\bar{\Omega})$ which is such that $\phi>0$ in $\bar{\Omega}$ and $\frac{\partial \phi}{ \partial \nu}=0$ on $\partial \Omega$.
\end{lemma}
\begin{proof}
    Thanks to \cite{Souplet-13}[Lemma 4.2] and Trace Sobolev embedding from $W^{1,1}(\Omega)$ into $L^1(\partial \Omega)$, one can find $c_1=c_1(\Omega)>0$ and $c_2>0$ such that 
    \begin{align*}
        \frac{\partial |\nabla \phi|^2}{\partial \nu } \leq c_1 |\nabla \phi|^2 \quad \text{on }\partial \Omega, 
    \end{align*}
    and 
    \begin{align*}
        \int_{\partial \Omega} |w| \leq c_2 \int_\Omega |\nabla w| +c_2 \int_\Omega |w| \quad \text{for all } w \in C^{1}(\bar{\Omega}).
    \end{align*}
    Therefore, we obtain that
    \begin{align} \label{Lbe.1}
           \int_{\partial \Omega} \phi^{-p-r}|\nabla \phi|^{2p-2} \cdot \frac{\partial |\nabla \phi|^2}{\partial \nu }&\leq c_1 \int_{\partial \Omega} \phi^{-p-r}|\nabla \phi|^{2p} \notag \\
           &\leq c_1c_2 \int_\Omega \left | {\nabla}  \left ( \phi^{-p-r}|\nabla \phi|^{2p} \right ) \right |+c_1c_2 \int_\Omega \phi^{-p-r}|\nabla \phi|^{2p}.
    \end{align}
    Using the identity $\nabla |\nabla \phi|^2= 2 D^2 \phi \cdot \nabla \phi$ and Young's inequality, it follows that
    \begin{align} \label{Lbe.2}
        c_1c_2 \int_\Omega \left | \nabla \left ( \phi^{-p-r}|\nabla \phi|^{2p} \right ) \right | &\leq 2pc_1c_2 \int_\Omega \phi^{-p-r} |\nabla \phi|^{2p-1} |D^2\phi| +c_1c_2(p+r) \int_\Omega \phi^{-p-r-1}|\nabla \phi|^{2p+1} \notag \\
        &\leq \epsilon \int_\Omega \phi^{-p-r}|\nabla \phi|^{2p-2} |D^2 \phi|^2 + \frac{p^2c_1^2c_2^2}{\epsilon} \int_\Omega \phi^{-p-r}|\nabla \phi|^{2p} \notag \\
        &\quad +c_1c_2(p+r) \int_\Omega \phi^{-p-r-1}|\nabla \phi|^{2p+1}.
    \end{align}
    Abbreviating $c_3:= c_1c_2+ \frac{p^2c_1^2c_2^2}{\epsilon} $ and $c_4:= c_1c_2(p+r)$, and applying Young's inequality with $\epsilon>0$, we derive that
    \begin{align} \label{Lbe.3}
        c_3 \int_\Omega \phi^{-p-r}|\nabla \phi|^{2p} &= \int_\Omega \left \{ \frac{\epsilon}{2} \phi^{-p-r-2} |\nabla \phi|^{2p+2} \right \}^{\frac{p}{p+1}} \cdot \left \{ \left ( \frac{2}{\epsilon}  \right )^{\frac{p}{p+1}}c_3 \phi^{\frac{p-r}{p+1}} \right \} \notag \\
        &\leq \frac{\epsilon}{2} \int_\Omega \phi^{-p-r-2}|\nabla \phi|^{2p+2} +\frac{2^p}{\epsilon^p}c_3^{p+1}\int_\Omega \phi^{p-r},
    \end{align}
    and 
    \begin{align} \label{Lbe.4}
        c_4\int_\Omega \phi^{-p-r-1}|\nabla \phi|^{2p+1} &= \int_\Omega \left \{ \frac{\epsilon}{2} \phi^{-p-r-2}|\nabla \phi|^{2p+2} \right \}^{\frac{2p+1}{2p+2}} \cdot \left \{ \left (\frac{2}{\epsilon} \right )^{\frac{2p+1}{2p+2}} c_4 \phi^{\frac{p-r}{2p+2}}\right \} \notag \\
        &\leq  \frac{\epsilon}{2} \int_\Omega \phi^{-p-r-2}|\nabla \phi|^{2p+2} + \left ( \frac{2}{\epsilon} \right )^{2p+1}c_4^{2p+2} \int_\Omega \phi^{p-r}.
    \end{align}
    Combining estimates from \eqref{Lbe.1} to \eqref{Lbe.4} yields
      \begin{align} \label{Lbe.5}
        \int_{\partial \Omega} \phi^{-p-r}|\nabla \phi|^{2p-2} \cdot \frac{\partial |\nabla \phi|^2}{\partial \nu }\leq \epsilon \int_\Omega \phi^{-p-r}|\nabla \phi|^{2p-2} |D^2 \phi|^2+ \epsilon \int_\Omega \phi^{-p-r-2}|\nabla \phi|^{2p+2} +c_5 \int_\Omega \phi^{p-r},
    \end{align}
    where $c_5= \frac{2^p}{\epsilon^p}c_3+\left ( \frac{2}{\epsilon} \right )^{2p+1}c_4^{2p+2}$. Thanks to {\eqref{Lm.2} and \eqref{Lm.3}}, one can find $c_6=c_6(p,r,n)>0$ such that 
    \begin{align*}
        \int_\Omega \phi^{-p-r}|\nabla \phi|^{2p-2} |D^2 \phi|^2+ \int_\Omega \phi^{-p-r-2}|\nabla \phi|^{2p+2} \leq c_6 \int_\Omega \phi^{-p-r+2}|\nabla \phi|^{2p-2} |D^2 \ln \phi|^2.
    \end{align*}
    Choosing $\epsilon = \frac{\eta}{c_6}$ and substituting this into \eqref{Lbe.5} proves \eqref{Lbe-1}, thereby completing the proof.
    
\end{proof}

\section{A priori estimates}\label{S3}

In this section, we establish several estimates for solutions to \eqref{1}, which will be later applied to derive the most essential estimate, Lemma \ref{L4}, allowing us to obtain {a uniform in time $L^{p}$ bound with some $p> \frac{n}{2}$ for $u$}. Let us commence with the following lemma, which enables us to control the quantity $\int_\Omega v^{p-r}$.

\begin{lemma} \label{Lp-v}
    For any $p>1$, $0<r<p-1$, $M>0$, $\chi>0$ and $\epsilon>0$, there exists $C=C(p,r,M, \epsilon, u_0,v_0)>0$ such that 
    \begin{align*}
        \frac{d}{dt} \int_\Omega v^{p-r} + M \int_\Omega v^{p-r} \leq \epsilon \int_\Omega \frac{u^{p+1}}{v^{r+1}}+C \qquad \text{for all }t\in (0,T_{\rm max}).
    \end{align*}
\end{lemma}
\begin{proof}
    By testing the second equation of \eqref{1} by $(p-r)v^{p-r-1}$, we obtain that 
    \begin{align}\label{Lp-v.1}
        \frac{d}{dt} \int_\Omega v^{p-r} &= -(p-r)(p-r-1) \int_\Omega v^{p-r-2}|\nabla v|^2 - (p-r) \int_\Omega v^{p-r} +(p-r)\int_\Omega u v^{p-r-1}.
    \end{align}
    Applying Young's inequality with $\epsilon>0$ yields 
    \begin{align*}
     (p-r)\int_\Omega u v^{p-r-1} &= \int_\Omega \left \{ \epsilon \frac{u^{p+1}}{v^{r+1}}  \right \}^{\frac{1}{p+1}} \cdot \left \{ \epsilon^{-\frac{1}{p}} (p-r)^{\frac{p+1}{p}} v^{p-r} \right \}^{\frac{p}{p+1}} \notag \\
     &\leq \epsilon  \int_\Omega \frac{u^{p+1}}{v^{r+1}}+c_1 \int_\Omega v^{p-r},
    \end{align*}
    where $c_1= \epsilon^{-\frac{1}{p}} (p-r)^{\frac{p+1}{p}}$. {This, together with \eqref{Lp-v.1} implies that 
    \begin{align}\label{Lp-v.3}
        \frac{d}{dt} \int_\Omega v^{p-r} + M\int_\Omega v^{p-r} &\leq   -(p-r)(p-r-1) \int_\Omega v^{p-r-2}|\nabla v|^2 + \epsilon \int_\Omega \frac{u^{p+1}}{v^{r+1}} \notag \\
        &\quad+( M + c_1-  {p+r} ) \int_\Omega v^{p-r}  \qquad \text{for all }t\in (0,T_{\rm max}).
    \end{align}
    }Thanks to Gagliardo–Nirenberg interpolation inequality, Young's inequality, and Lemma \ref{L1}, it follows that 
    \begin{align}\label{Lp-v.2}
        \left ( M + c_1-  {p+r} \right )_+\int_\Omega v^{p-r} &\leq  c_2  \left ( \int_\Omega |\nabla v^\frac{p-r}{2}|^2 \right )^\theta \left (\int_\Omega v \right )^{(p-r)(1-\theta)} +c_2  \left (\int_\Omega v \right )^{p-r} \notag \\ 
        &\leq (p-r)(p-r-1) \int_\Omega v^{p-r-2}|\nabla v|^2 +c_3,
    \end{align}
    where $\theta = \frac{\frac{p-r}{2}-\frac{1}{2}}{\frac{p-r}{2}+\frac{1}{n}-\frac{1}{2}} \in (0,1)$, $c_2>0$, and $c_3>0$. Finally, combining \eqref{Lp-v.3} and \eqref{Lp-v.2}, we arrive at 
    {
\begin{align*}
    \frac{d}{dt}\int_\Omega v^{p-r} +M \int_\Omega v^{p-r} \leq \epsilon \int_\Omega \frac{u^{p+1}}{v^{r+1}} +c_3 \qquad \text{for all } t\in (0,T_{\rm max}),
\end{align*}
    }
    which completes the proof.
    
\end{proof}

    \begin{lemma} \label{L2}
    Let $p>1$, $r= \frac{p-1}{2}$ and { $\chi>0$}. Then there exists $C=C(p,n)>0$ such that the following holds
        \begin{align}
        \frac{d}{dt}\int_\Omega v^{-p-r}|\nabla v|^{2p} &\leq -p \int_\Omega v^{-p-r+2}|\nabla v|^{2p-2}|D^2 \ln v|^2-(p-r)\int_\Omega v^{-p-r}|\nabla v|^{2p} \notag \\
        &\quad-\frac{p-1}{4}\int_\Omega v^{-p-r-2}|\nabla v|^{2p+2} +C\int_\Omega v^{p-r} \notag \\
        &\quad+2p(2p-2+\sqrt{ n}) \int_\Omega uv^{-p-r}|\nabla v|^{2p-2}|D^2v| \notag \\
        &\quad +(2p-1)(p+r) \int_\Omega v^{-p-r-1}|\nabla v|^{2p } u
   \end{align}
   for any $t\in (0,T_{\rm max})$.
    \end{lemma}
    \begin{proof}
        Direct calculations show that
        \begin{align} \label{L2.1}
            \frac{d}{dt}\int_\Omega v^{-p-r}|\nabla v|^{2p}&= 2p \int_\Omega v^{-p-r}|\nabla v|^{2p-2}\nabla v \cdot \nabla v_t -(p+r)\int_\Omega v^{-p-r-1}|\nabla v|^{2p}v_t \notag \\
            &=I+J.
        \end{align}
    By using the second equation of \eqref{1} as well as the identities $\nabla v \cdot \nabla \Delta v = \frac{1}{2} \Delta |\nabla v|^2 - |D^2 v|^2$ and $\nabla |\nabla v|^2 = 2 D^2v \cdot \nabla v$, and performing several integration by parts, we obtain that 
    \begin{align} \label{L2.2}
        I&= 2p\int_\Omega v^{-p-r}|\nabla v|^{2p-2}\nabla v \cdot \nabla \left ( \Delta v- v+u \right ) \notag \\
        &=2p\int_\Omega v^{-p-r}|\nabla v|^{2p-2}\nabla v \cdot \nabla \Delta v -2p \int_\Omega  v^{-p-r}|\nabla v|^{2p}+2p\int_\Omega v^{-p-r}|\nabla v|^{2p-2}\nabla v \cdot \nabla u \notag \\
        &= p\int_\Omega v^{-p-r}|\nabla v|^{2p-2} \Delta|\nabla v|^2- 2p \int_\Omega v^{-p-r}|\nabla v|^{2p-2}|D^2v|^2  -2p \int_\Omega  v^{-p-r}|\nabla v|^{2p} \notag \\
        &\quad +2p(p+r)\int_\Omega v^{-p-r-1}|\nabla v|^{2p}u -2p(p-1) \int_\Omega uv^{-p-r} |\nabla v|^{2p-4} \nabla |\nabla v|^2 \cdot \nabla v  \notag \\
        &\quad-2p \int_\Omega u v^{-p-r}|\nabla v|^{2p-2}\Delta v \notag \\
        &= -p(p-1)\int_\Omega v^{-p-r}|\nabla v|^{2p-4}|\nabla |\nabla v|^2|^2+p(p+r)\int_\Omega v^{-p-r-1}|\nabla v|^{2p-2} \nabla v \cdot \nabla |\nabla v|^2 \notag \\
        &\quad+p \int_{\partial \Omega} v^{-p-r}|\nabla v|^{2p-2}\cdot \frac{\partial |\nabla v|^2}{\partial \nu}- 2p \int_\Omega v^{-p-r}|\nabla v|^{2p-2}|D^2v|^2  -2p \int_\Omega  v^{-p-r}|\nabla v|^{2p} \notag \\
        &\quad +2p(p+r)\int_\Omega v^{-p-r-1}|\nabla v|^{2p}u -4p(p-1) \int_\Omega uv^{-p-r} |\nabla v|^{2p-4} \nabla v \cdot D^2 v \cdot \nabla v  \notag \\
        &\quad-2p \int_\Omega u v^{-p-r}|\nabla v|^{2p-2}\Delta v,
    \end{align}
    and 
    \begin{align} \label{L2.3}
        J&= -(p+r)\int_\Omega v^{-p-r-1}|\nabla v|^{2p} \Delta v +(p+r)\int_\Omega v^{-p-r}|\nabla v|^{2p} -(p+r)\int_\Omega v^{-p-r-1}|\nabla v|^{2p } u \notag \\
        &=-(p+r)(p+r+1) \int_\Omega v^{-p-r-2}|\nabla v|^{2p+2}+p(p+r) \int_\Omega v^{-p-r-1}|\nabla v|^{2p-2} \nabla |\nabla v|^2 \cdot \nabla v \notag \\
        &\quad +(p+r)\int_\Omega v^{-p-r}|\nabla v|^{2p} -(p+r)\int_\Omega v^{-p-r-1}|\nabla v|^{2p } u.
    \end{align}
    Thanks to \eqref{L2.1}, \eqref{L2.2}, \eqref{L2.3} and the identity that 
    \begin{align*}
        |D^2v|^2= v^2|D^2 \ln v|^2+\frac{1}{v} \nabla v \cdot \nabla |\nabla v|^2 -\frac{|\nabla v|^4}{v^2},
    \end{align*}
    we can write 
    \begin{align*}
        \frac{d}{dt}\int_\Omega v^{-p-r}|\nabla v|^{2p}&= -2p \int_\Omega v^{-p-r+2}|\nabla v|^{2p-2}|D^2 \ln v|^2-p(p-1)\int_\Omega v^{-p-r}|\nabla v|^{2p-4}|\nabla |\nabla v|^2|^2 \notag \\
        &\quad+2p(p+r-1)\int_\Omega v^{-p-r-1}|\nabla v|^{2p-2} \nabla |\nabla v|^2 \cdot \nabla v -(p-r)\int_\Omega v^{-p-r}|\nabla v|^{2p}\notag \\
        &\quad- \left [ (p+r)(p+r+1)-2p \right]\int_\Omega v^{-p-r-2}|\nabla v|^{2p+2}+(2p-1)(p+r) \int_\Omega v^{-p-r-1}|\nabla v|^{2p } u \notag \\
        &\quad-4p(p-1) \int_\Omega uv^{-p-r} |\nabla v|^{2p-4} \nabla v \cdot D^2 v \cdot \nabla v  -2p \int_\Omega u v^{-p-r}|\nabla v|^{2p-2}\Delta v \notag \\
        &\quad+p \int_{\partial \Omega} v^{-p-r}|\nabla v|^{2p-2}\cdot \frac{\partial |\nabla v|^2}{\partial \nu}.
    \end{align*}
    Applying the inequality that $|\Delta v|\leq \sqrt{n}|D^2v|$ and Young's inequality  yields 
    \begin{align} \label{L2.3'}
         \frac{d}{dt}\int_\Omega v^{-p-r}|\nabla v|^{2p} &\leq -2p \int_\Omega v^{-p-r+2}|\nabla v|^{2p-2}|D^2 \ln v|^2-(p-r)\int_\Omega v^{-p-r}|\nabla v|^{2p} \notag \\
         &\quad- \left [ (p+r)(p+r+1)-2p-\frac{p(p+r-1)^2}{p-1} \right ]\int_\Omega v^{-p-r-2}|\nabla v|^{2p+2} \notag \\
         &\quad+2p(2p-2+\sqrt{ n}) \int_\Omega uv^{-p-r}|\nabla v|^{2p-2}|D^2v| \notag \\
         &\quad+(2p-1)(p+r) \int_\Omega v^{-p-r-1}|\nabla v|^{2p } u+p \int_{\partial \Omega} v^{-p-r}|\nabla v|^{2p-2}\cdot \frac{\partial |\nabla v|^2}{\partial \nu}.
    \end{align}
   Applying Lemma \ref{Lbe}, noting that  $(p+r)(p+r+1)-2p-\frac{p(p+r-1)^2}{p-1}= \frac{p-1}{4}$ since $r=\frac{p-1}{2}$, and using \eqref{L2.3'}, we arrive at
      \begin{align}
        \frac{d}{dt}\int_\Omega v^{-p-r}|\nabla v|^{2p} &\leq -p \int_\Omega v^{-p-r+2}|\nabla v|^{2p-2}|D^2 \ln v|^2-(p-r)\int_\Omega v^{-p-r}|\nabla v|^{2p} \notag \\
        &\quad-\frac{p-1}{4}\int_\Omega v^{-p-r-2}|\nabla v|^{2p+2} +c_1\int_\Omega v^{p-r} \notag \\
        &\quad+2p(2p-2+\sqrt{ n}) \int_\Omega uv^{-p-r}|\nabla v|^{2p-2}|D^2v| \notag \\
        &\quad +(2p-1)(p+r) \int_\Omega v^{-p-r-1}|\nabla v|^{2p } u,
   \end{align}
   for some $c_1>0$. The proof is now complete.
        
    \end{proof}

The following lemma, which justifies the consideration of the energy functional \( F_\lambda(u,v) \), indicates that the quantity \( \int_\Omega \frac{|\nabla v|^{2p+2}}{v^{p+r+2}} \)  can be controlled by \( \int_\Omega \frac{u^{p+1}}{v^{r+1}} \).

\begin{lemma}\label{L2'}
Let $p\in\bigl(\frac{n}{2},\,n\bigr)$, set $r=\frac{p-1}{2}$, and let $\chi>0$. 
Then there exists a constant $C=C(p,n)>0$ such that the inequality
\begin{align}\label{L2'-1}
\frac{d}{dt}\int_\Omega \frac{|\nabla v|^{2p}}{v^{p+r}}
+ L_1 \int_\Omega \frac{|\nabla v|^{2p}}{v^{p+r}}
+ L_2 \int_\Omega \frac{|\nabla v|^{2p+2}}{v^{p+r+2}}
\leq L_3 \int_\Omega \frac{u^{p+1}}{v^{r+1}}
+ C \int_\Omega v^{p-r}
\end{align}
holds for all $t\in(0,T_{\max})$, where $L_1=\frac{n+2}{8}$ and the constants $L_2$ and $L_3$ are specified in Remark~\ref{Rm}.
\end{lemma}

    \begin{proof}
        Using Lemma \ref{L2} and noting that $\frac{n}{2}<p<n$ yields 
        \begin{align} \label{L2'.1}
             \frac{d}{dt} \int_\Omega \frac{|\nabla v|^{2p}}{v^{p+r}}&+ \frac{n+2}{8} \int_\Omega \frac{|\nabla v|^{2p}}{v^{p+r}} + \frac{n-2}{8}\int_\Omega \frac{|\nabla v|^{2p+2}}{v^{p+r+2}}+ \frac{n}{2}\int_\Omega v^{-p-r+2}|\nabla v|^{2p-2}|D^2 \ln v|^2 \notag \\
             &\leq c_1 \int_\Omega v^{p-r}+ c_2\int_\Omega uv^{-p-r}|\nabla v|^{2p-2}|D^2v|+c_3 \int_\Omega v^{-p-r-1}|\nabla v|^{2p } u,
        \end{align}
       where $c_1=c_1(p,n)>0$, $c_2= 2n(2n-2+\sqrt{n})$, and $c_3= \frac{(2n-1)(3n-1)}{2}$. Thanks to Lemma \ref{Lw2}, we obtain that 
        \begin{align} \label{L2'.2}
            \int_\Omega \frac{|\nabla v|^{2p+2}}{v^{p+r+2}} &\leq \left ( 4 \left ( \frac{p+r+1+\sqrt{n}}{p-r+1} \right )^2+2\right )\int_\Omega v^{-p-r+2}|\nabla v|^{2p-2}|D^2 \ln v|^2 \notag \\
            &\leq 102 \int_\Omega v^{-p-r+2}|\nabla v|^{2p-2}|D^2 \ln v|^2,
        \end{align}
        where the last inequality follows from
        \begin{align}\label{L2'.3}
             \frac{p+r+1+\sqrt{n}}{p-r+1} =\frac{3p+1+2\sqrt{n}}{p+3} <5,
        \end{align}
        {which holds for all $p > \sqrt{n} - 7$; this condition is satisfied since $p > \frac{n}{2} > \sqrt{n} - 7$.}
        Similarly, we also have
        \begin{align} \label{L2'.4}
            \int_\Omega v^{-p-r}|\nabla v|^{2p-2} |D^2 v|^2 &\leq \left ( \frac{p-r+1}{p-r-1} \left ( \frac{p+r+1+\sqrt{n}}{p-r+1} \right )^2 +1 \right )\int_\Omega v^{-p-r+2}|\nabla v|^{2p-2}|D^2 \ln v|^2 \notag \\
            &\leq 226\int_\Omega v^{-p-r+2}|\nabla v|^{2p-2}|D^2 \ln v|^2,
        \end{align}
        {where the last inequality follows from \eqref{L2'.3} together with the estimate
\[
    \frac{p-r+1}{p-r-1}
    = \frac{p+3}{p-1}
    < 9,
\]
which is valid since $p > \frac{n}{2} \ge \frac{3}{2}$.}
    In light of Young's inequality, we derive that 
    \begin{align}\label{L2'.5}
        c_3 \int_\Omega v^{-p-r-1}|\nabla v|^{2p}u &= {\int_\Omega \left [ \left ( \frac{n-2}{16} \right )^{\frac{p}{p+1}} \frac{|\nabla v|^{2p}}{v^{\frac{p(p+r+2)}{p+1}}} \right ] \cdot \left [ c_3 \left ( \frac{16}{n-2}\right )^{\frac{p}{p+1}} \frac{u}{v^{\frac{r+1}{p+1}}} \right ] } \notag \\
        & {\leq \frac{p(n-2)}{16(p+1)}\int_\Omega \frac{|\nabla v|^{2p+2}}{v^{p+r+2}} + \left ( \frac{16}{n-2} \right )^p\frac{c_3^{p+1}}{p+1} \int_\Omega \frac{u^{p+1}}{v^{r+1}} } \notag\\
        &\leq \frac{n-2}{16}\int_\Omega \frac{|\nabla v|^{2p+2}}{v^{p+r+2}} + \left ( \frac{16}{n-2} \right )^pc_3^{p+1} \int_\Omega \frac{u^{p+1}}{v^{r+1}} \notag \\
        &\leq \frac{n-2}{16}\int_\Omega \frac{|\nabla v|^{2p+2}}{v^{p+r+2}} + 16^nc_3^{n+1} \int_\Omega \frac{u^{p+1}}{v^{r+1}}.
    \end{align}        
    Applying Hölder's inequality, \eqref{L2'.2}, \eqref{L2'.4} and Young's inequality and {noting that $p<n$ and $ \sqrt{226} (102)^{\frac{p-1}{2(p+1)}} < 152$, we obtain that } 
        \begin{align}\label{L2'.6}
             c_2\int_\Omega uv^{-p-r}|\nabla v|^{2p-2}|D^2v| &\leq c_2 \left ( \int_\Omega v^{-p-r}|\nabla v|^{2p-2}|D^2 v|^2 \right )^\frac{1}{2} \left ( \int_\Omega \frac{|\nabla v|^{2p+2}}{v^{p+r+2}} \right)^\frac{p-1}{2(p+1)} \left ( \int_\Omega \frac{u^{p+1}}{v^{r+1}}\right )^{\frac{1}{p+1}} \notag \\
            & {\leq \sqrt{226} (102)^{\frac{p-1}{2(p+1)}}c_2\left ( \int_\Omega v^{-p-r+2}|\nabla v|^{2p-2}|D^2 \ln v|^2 \right )^{\frac{p}{p+1}} \left ( \int_\Omega \frac{u^{p+1}}{v^{r+1}}\right )^{\frac{1}{p+1}} }  \notag \\
             &\leq 152 c_2 \left ( \int_\Omega v^{-p-r+2}|\nabla v|^{2p-2}|D^2 \ln v|^2 \right )^{\frac{p}{p+1}} \left ( \int_\Omega \frac{u^{p+1}}{v^{r+1}}\right )^{\frac{1}{p+1}} \notag \\
             &{\leq \frac{np}{2(p+1)} \int_\Omega v^{-p-r+2}|\nabla v|^{2p-2}|D^2 \ln v|^2+ \frac{(152c_2)^{p+1}}{p+1} \int_\Omega \frac{u^{p+1}}{v^{r+1}} } \notag \\
             &\leq \frac{n}{2}\int_\Omega v^{-p-r+2}|\nabla v|^{2p-2}|D^2 \ln v|^2 + (152c_2)^{n+1}\int_\Omega \frac{u^{p+1}}{v^{r+1}}.  
        \end{align}
   { Combining \eqref{L2'.1}, \eqref{L2'.5} and \eqref{L2'.6}, we infer that 
    \begin{align*}
           \frac{d}{dt} \int_\Omega \frac{|\nabla v|^{2p}}{v^{p+r}}+\frac{n+2}{8}\int_\Omega \frac{|\nabla v|^{2p}}{v^{p+r}}+  \frac{n-2}{16}\int_\Omega \frac{|\nabla v|^{2p+2}}{v^{p+r+2}} &\leq \left ( 16^n c_3^{n+1} +(152c_2)^{n+1} \right ) \int_\Omega \frac{u^{p+1}}{v^{r+1}} \notag \\
           &\quad+c_1\int_\Omega v^{p-r} \qquad \text{for all }t\in (0,T_{\rm max}).
    \end{align*} 
    This implies \eqref{L2'-1} and thereby completes the proof.}
    \end{proof}

The following identity plays a crucial role in establishing uniform boundedness in \( L^p(\Omega) \); for a detailed proof, we refer the reader to \cite{Winkler+2011}[Lemma 2.3].

\begin{lemma} \label{L5}
    Let $p\in \mathbb{R}$, $r \in  \mathbb{R}$ and {$\chi>0$}. Then the following identity holds for any $t\in (0,T_{\rm max})$:
    \begin{align*}
        \frac{d}{dt}\int_\Omega u^p v^{-r} + r\int_\Omega u^{p+1}v^{-r-1} &= -p(p-1)\int_\Omega u^{p-2}v^{-r}|\nabla u|^2 - \left[ r(r+1)+pr\chi \right ] \int_\Omega u^p v^{-r-2}|\nabla v|^2 \notag \\
        &\quad+  \left[ 2pr+p(p-1)\chi \right ] \int_\Omega u^{p-1}v^{-r-1} \nabla u \cdot \nabla v+r \int_\Omega u^pv^{-r}.
        \end{align*}    
\end{lemma}

As a consequence, we have the following key estimate which will play an important role in deriving the boundedness of $F_\lambda(u,v)$. 
    
\begin{lemma} \label{L3}
    {Let $r= \frac{p-1}{2}$ and $\chi>0$}. There exist a positive constant $C=C(p,n)$ such that for any $p \in \left ( \frac{n}{2}, n \right )$ the following holds
    \begin{align} \label{L3-1}
        \frac{d}{dt} \int_\Omega u^p v^{-r}+ \int_\Omega u^p v^{-r} + L_4 \int_\Omega u^{p+1}v^{-r-1} &\leq L_5(\chi^2p-1)_+\int_\Omega u^pv^{-p-r-2}|\nabla v|^2 \notag \\
        &\quad+C \int_\Omega v^{p-r} \qquad \text{for all }t\in (0,T_{\rm max}),
    \end{align}
   {where $L_4$ and $L_5$ are given in Remark \ref{Rm}.}
\end{lemma}

\begin{proof}
    From Lemma \ref{L5}, we have that 
     \begin{align} \label{L5.1}
        I:=\frac{d}{dt}\int_\Omega u^p v^{-r} + r\int_\Omega u^{p+1}v^{-r-1} &= -p(p-1)\int_\Omega u^{p-2}v^{-r}|\nabla u|^2 - \left[ r(r+1)+pr\chi \right ] \int_\Omega u^p v^{-r-2}|\nabla v|^2 \notag \\
        &\quad+  \left[ 2pr+p(p-1)\chi \right ] \int_\Omega u^{p-1}v^{-r-1} \nabla u \cdot \nabla v+r \int_\Omega u^pv^{-r}.
        \end{align}
    We apply Young's inequality to obtain that 
    \begin{align*}
        \left[ 2pr+p(p-1)\chi \right ] \int_\Omega u^{p-1}v^{-r-1} \nabla u \cdot \nabla v &\leq p(p-1)\int_\Omega u^{p-2}v^{-r}|\nabla u|^2 \notag \\
        &\quad+ \frac{\left [2pr+p(p-1)\chi \right ]^2}{4p(p-1)}\int_\Omega u^p v^{-r-2}|\nabla v|^2.
    \end{align*}
    Therefore, \eqref{L5.1} yields 
    \begin{align}\label{L5.2}
         I &\leq  -  h(p.r,\chi) \int_\Omega u^p v^{-r-2}|\nabla v|^2 +r \int_\Omega u^p v^{-r}, 
    \end{align}
    where 
    \begin{align*}
        h(p,r,\chi)= r(r+1)+pr\chi-\frac{\left [2pr+p(p-1)\chi \right ]^2}{4p(p-1)} = -\frac{(p-1)(\chi^2p-1)}{4}.
    \end{align*}
    By Young's inequality, one can find $c_1=c_1(p)>0$ such that
    \begin{align}\label{L5.3}
       (r+1) \int_\Omega u^pv^{-r} \leq \frac{r}{2}\int_\Omega u^{p+1}v^{-r-1}+ c_1 \int_\Omega v^{p-r}.
    \end{align}
 { By combining \eqref{L5.1}--\eqref{L5.3} and noting that $\frac{r}{2} > \frac{n-2}{8}$ and $\frac{p-1}{4}> \frac{n-2}{8}$, we arrive at 
    \begin{align*}
        \frac{d}{dt}\int_\Omega u^pv^{-r}+\int_\Omega u^pv^{-r} + \frac{n-2}{8} \int_\Omega u^{p+1}v^{-r-1} &\leq  \frac{(n-2)(\chi^2p-1)_+}{8}\int_\Omega u^p v^{-r-2}|\nabla v|^2 \notag \\
        &\quad+c_1 \int_\Omega v^{p-r} \qquad \text{for all }t\in (0,T_{\rm max}).
    \end{align*}
This proves \eqref{L3-1} and thereby concludes the proof. }
        
\end{proof}

\section{Proof of the main result}\label{S4}
Let us begin this section by showing that the energy function $F_{\lambda}(u,v)$, with an appropriate choice of $\lambda$ is uniformly bounded.  
{\begin{lemma} \label{L4}
Let $r = \frac{p-1}{2}$ and
\(
\delta =  \frac{L_2 L_4}{2 L_3 L_5},
\)
where $L_i$, with $i=2,3,4,5$ are given in Remark \ref{Rm}.
For any
\[
0 < \chi < \sqrt{\frac{2(1+\delta)}{n}}
\quad \text{and} \quad
1 < p < \min\left\{ \frac{1+\delta}{\chi^2}, \, n \right\},
\]
there exists a constant $C = C(u_0, v_0, n, \Omega, p) > 0$ such that
\begin{align} \label{L4-1}
    \int_\Omega v^{p-r}
    + \int_\Omega u^p v^{-r}
    + \int_\Omega v^{-p-r} |\nabla v|^{2p}
    \le C
    \qquad \text{for all } t \in (0, T_{\max}).
\end{align}
\end{lemma}

\begin{proof}
We first assume that $p > \frac{n}{2}$, and define $y(t) := \int_\Omega u^p v^{-r} + \lambda \int_\Omega v^{-p - r} |\nabla v|^{2p}$ with $\lambda = \frac{L_4}{2L_3}$.
 From Lemma \ref{L2'} and Lemma \ref{L3}, we find that
    \begin{align} \label{L4.1}
        y'(t) +c_1 y(t)+ \frac{L_4L_2}{2L_3}\int_\Omega \frac{|\nabla v|^{2p+2}}{v^{p+r+2}} +  \frac{L_4}{2} \int_\Omega u^{p+1}v^{-r-1} &\leq L_5\left ( \chi^2p-1\right )_+\int_\Omega u^pv^{-r-2}|\nabla v|^2 \notag \\
        &\quad+c_2\int_\Omega v^{p-r}.
    \end{align}
     $c_1= \min \left \{ 1,\frac{L_1L_4}{2L_3}\right \}$ with $L_1$ is given in Lemma \ref{L2'} and $c_2>0$. 
     {Since $\delta = \frac{L_2 L_4}{2 L_3 L_5}$, we infer that for any
\[
\frac{n}{2} < p < \frac{1+\delta}{\chi^2},
\]
it holds that
\[
L_5 \big( \chi^2 p - 1 \big)_+
\le \min\left\{ \frac{L_2 L_4}{2 L_3}, \, \frac{L_4}{4} \right\}.
\]
}
     Using this and applying Young's inequality yields 
    \begin{align} \label{L4.2}
        L_5\left ( \chi^2p-1\right )_+\int_\Omega u^pv^{-r-2}|\nabla v|^2 &= L_5 \left ( \chi^2p-1\right )_+\int_\Omega \frac{u^p}{v^{\frac{p(r+1)}{p+1}}}\cdot \frac{|\nabla v|^2}{ v^{\frac{p+r+2}{p+1}}} \notag \\
        &\leq {\frac{pL_5\left ( \chi^2p-1\right )_+}{p+1}\int_\Omega \frac{u^{p+1}}{v^{r+1}}+ \frac{L_5\left ( \chi^2p-1\right )_+}{p+1}\int_\Omega \frac{|\nabla v|^{2p+2}}{v^{p+r+2}}  } \notag \\
        &\leq  L_5\left ( \chi^2p-1\right )_+\int_\Omega \frac{u^{p+1}}{v^{r+1}} +L_5\left ( \chi^2p-1\right )_+\int_\Omega \frac{|\nabla v|^{2p+2}}{v^{p+r+2}} \notag \\
        & {\leq \frac{L_4}{4}\int_\Omega \frac{u^{p+1}}{v^{r+1}} + \frac{L_4L_2}{2L_3} \int_\Omega \frac{|\nabla v|^{2p+2}}{v^{p+r+2}}.}
    \end{align}
  Combining \eqref{L4.1} and \eqref{L4.2}, it follows that
    \begin{align}\label{L4.3}
        y'(t)+c_1 y(t) +\frac{L_4}{4}\int_\Omega \frac{u^{p+1}}{v^{r+1}} \leq c_2 \int_\Omega v^{p-r} \qquad \text{for all }t \in (0, T_{\rm max}).
    \end{align}
    Applying Lemma \ref{Lp-v} with $M=c_2+c_1$ and $\epsilon =\frac{L_4}{4}$, we derive that
    \begin{align}\label{L4.4}
        \frac{d}{dt}\int_\Omega v^{p-r} + (c_2+c_1)\int_\Omega v^{p-r} \leq \frac{L_4}{4}\int_\Omega \frac{u^{p+1}}{v^{r+1}}+ c_3  \qquad \text{for all }t \in (0, T_{\rm max}),
    \end{align}
    for some $c_3>0$. Combining \eqref{L4.3} and \eqref{L4.4}, we arrive at 
    \begin{align*}
        \frac{d}{dt} F_{\lambda}(u,v) +c_1F_{\lambda}(u,v) \leq c_3 \qquad \text{for all }t\in (0,T_{\rm max}),
    \end{align*}
    which combines with Gronwall's inequality implies \eqref{L4-1} for any $\frac{n}{2}<p<\frac{1+\delta}{\chi^2} $.  Now, for any $1<q<p$, applying Hölder's inequality and Lemma \ref{L1}, it follows that
    \begin{align*}
        \int_\Omega u^q v^{-\frac{q-1}{2}}+ \int_\Omega v^{-\frac{3q-1}{2}}|\nabla v|^{2q} &\leq \left ( \int_\Omega u^pv^{-\frac{p-1}{2}} \right )^{\frac{q}{p}} \left ( \int_\Omega v^{\frac{1}{2}} \right )^{1-\frac{q}{p}}+\left ( \int_\Omega v^{-\frac{3p-1}{2}} |\nabla v|^{2p} \right )^{\frac{q}{p}} \left ( \int_\Omega v^{\frac{1}{2}} \right )^{1-\frac{q}{p}} \notag \\        
        & \leq |\Omega|^{\frac{1}{2}-\frac{q}{2p}}\left ( \int_\Omega v \right )^{\frac{1}{2}-\frac{q}{2p}}\left \{ \left ( \int_\Omega u^pv^{-\frac{p-1}{2}} \right )^{\frac{q}{p}} + \left ( \int_\Omega v^{-\frac{3p-1}{2}} |\nabla v|^{2p} \right )^{\frac{q}{p}}\right \}  .
    \end{align*}
    Therefore, \eqref{L4-1} also holds for any $1<p \leq \frac{n}{2}$. The proof is now complete.
\end{proof}
As a consequence, we can now derive an $L^{\frac{n}{2}+}$ for $u$ as follows:
\begin{corollary} \label{C1}
    Assume that $\chi < \sqrt{\frac{2(1+\delta)}{n}}$ where $\delta$ is given in Lemma \ref{L4}. Then there exists $p> \frac{n}{2}$ such that 
    \begin{align*}
      \sup_{t\in (0,T_{\rm max})}  \int_\Omega u^p(\cdot,t) <\infty.
    \end{align*}
\end{corollary}
\begin{proof}
    Applying Lemma \ref{L4} yields there exists $p_0\in \left (\frac{n}{2},n \right )$ such that
    \begin{align} \label{C1.1}
        \sup_{t \in (0,T_{\rm max})} \int_\Omega u^{p_0} v^{-r} := c_1<\infty,
    \end{align}
    where $r=\frac{p_0-1}{2}$.
    For any $T <T_{\rm max}$, setting 
    \begin{align*}
        M(T):= \sup_{t\in (0,T)}  \int_\Omega u^p(\cdot,t) 
    \end{align*}
    then we see from Lemma \ref{local} that $M(T)< \infty$. Since $p_0> \frac{n}{2}$ implies that $\frac{p_0+1}{2(n-p_0+1)}>\frac{n}{2}$, we can choose $p \in \left ( \frac{n}{2}, \min \left \{ p_0, \frac{p_0+1}{2(n-p_0+1) } \right \}\right ) $. Employing Hölder's inequality, we obtain 
    \begin{align*}
        \int_\Omega u^p \leq \left ( \int_\Omega u^{p_0}v^{-r} \right )^{\frac{p}{p_0}} \cdot \left ( \int_\Omega v^{\frac{pr}{p_0-p}} \right )^{\frac{p_0-p}{p_0}}.
    \end{align*}
   This, together with \eqref{C1.1}, implies that 
    \begin{align} \label{C1.2}
        \left \| u(\cdot,t) \right \|_{L^p(\Omega)} \leq c_1^{\frac{1}{p_0}}  \left \| v(\cdot,t) \right \|_{L^{\frac{pr}{p_0-p}}(\Omega)}^{\frac{r}{p_0}} \qquad \text{for all }t \in (0,T_{\rm max}).
    \end{align}
    Since $\frac{n}{2} \left ( \frac{1}{p}- \frac{p_0-p}{pr} \right )<1$ when $p<\frac{p_0+1}{2(n-p_0+1)}$, we apply Lemma \ref{v} to obtain that 
    \begin{align*}
        \sup_{t\in (0,T)} \left \| v(\cdot,t) \right \|_{L^{\frac{pr}{p_0-p}}(\Omega)} \leq c_2 \left (1+  \sup_{t\in (0,T)}  \int_\Omega u^p(\cdot,t) \right ),
    \end{align*}
    for some $c_2>0$ independent of $T$. This, conjunction with \eqref{C1.2} deduces that 
    \begin{align*}
        M(T) \leq c_3 \left (1+M(T)^{\frac{r}{p_0}} \right ) \qquad \text{for any }T<T_{\rm max},
    \end{align*}
    for some $c_3>0$ independent of $T$. Since $\frac{r}{p_0}<1$, this entails that  $M(T)\leq c_4$, for some $c_4>0$ independent of $T$, which further implies that $u \in L^\infty \left ( (0,T_{\rm max}); L^p(\Omega) \right )$.
\end{proof}
The following lemma enables us to derive \( L^\infty(\Omega \times (0, T_{\rm max})) \) bounds for the solution under the assumption that \( u \in L^\infty((0, T_{\rm max}); L^{p}(\Omega)) \) for some $p>\frac{n}{2}$.

\begin{lemma} \label{Linf}
    Let $\chi>0$, and suppose that $u \in L^\infty \left ((0,T_{\rm max});L^p(\Omega) \right )$ for some $p>\frac{n}{2}$. Then $u \in L^\infty \left ((0,T_{\rm max});L^\infty(\Omega) \right )$. 
\end{lemma}

\begin{proof}
    Thanks to a uniform lower bound for $v$ as established Lemma \ref{low}, we follow the standard heat semi-group technique (see Lemma 3.4 in \cite{Winkler+2011}) with the assistance of Lemma \ref{v}(ii) to prove that  $u \in L^\infty \left ((0,T_{\rm max});L^\infty(\Omega) \right )$.   
\end{proof}
We are now in a position to prove our main result.
\begin{proof}[Proof of Theorem \ref{thm}]
    Applying Corollary \ref{C1}, Lemma \ref{Linf}, Lemma \ref{v} and the extensibility property of solution \eqref{ext} as established in Lemma \ref{local} immediately proves the desired result.
\end{proof}

\section*{Acknowledgment}
The author gratefully thanks the anonymous referees for their valuable comments and suggestions. This work was supported by the Hangzhou Postdoctoral Research Grant (No.~207010136582503) and National Postdoctoral Talent Program.

    \end{document}